\documentclass{amsart}

\usepackage{amsmath,amssymb,mathrsfs,amsthm,graphicx,colonequals,hyperref}

\hypersetup{colorlinks=true,urlcolor=magenta,citecolor=magenta,linkcolor=blue}

\newcommand{\Gal}{\operatorname{Gal}}
\newcommand{\tr}{\operatorname{tr}}

\newcommand{\R}{\operatorname{\mathbf{R}}}
\newcommand{\Q}{\mathbf{Q}}
\newcommand{\ddef}{\colonequals}
\newcommand{\disc}{\operatorname{disc}}

\theoremstyle{plain}
\newtheorem{thm}[equation]{Theorem}
\newtheorem{lem}[equation]{Lemma}

\newtheorem{prop}[equation]{Proposition}

\theoremstyle{remark}
\newtheorem{rmk}[equation]{Remark}

\newtheorem{exm}[equation]{Example}

\begin{document}

\title{Hasse-Witt invariants for trace forms of Jacobi polynomials}

\author{John Cullinan, Farshid Hajir, Elisabeth Young}

\address{Department of Mathematics, Bard College, Annandale-On-Hudson, NY 12504, USA}
\email{cullinan@bard.edu}
\urladdr{\url{http://faculty.bard.edu/cullinan/}}

\address{Department of Mathematics and Statistics, University of Massachusetts Amherst, Amherst, MA 01003, USA}
\email{hajir@math.umass.edu}

\address{Department of Mathematics, Bard College, Annandale-On-Hudson, NY 12504, USA}
\email{ey0200@bard.edu}

\keywords{Jacobi Polynomial, Hankel Determinant, Hasse-Witt Invariant}
\subjclass{11R32, 15A15}

\begin{abstract}
In \cite{feit}, Feit used the Generalized Laguerre Polynomials (GLP) to prove that the groups $\widetilde{A}_{5}$ and $\widetilde{A}_{7}$ occur as Galois groups over $\Q$.  Hajir, in \cite{hajir} extended these results to prove that $\widetilde{A}_{n}$ is Galois over $\Q$ whenever $n \equiv 1 \pmod{8}$.  A key ingredient of both proofs is the explicit determination of the Hasse-Witt invariant of (the diagonalization of) the trace form of the root fields of the GLP, which relies on the calculation of a certain determinant, $\Delta_t$.  The explicit formula for $\Delta_t$ used in \cite{feit} and \cite{hajir} was derived using properties specific to the GLP which do not generalize to other polynomials.

In this paper we revisit Feit's original calculation of $\Delta_t$ and situate it in the context of Hankel determinants.  We give an alternate derivation of $\Delta_t$ using standard combinatorial arguments and then apply these results to the Jacobi polynomials, a two-parameter family of orthogonal polynomials encompassing the GLP as a special case.  We compute an explicit formula for the $\Delta_t$ of the Jacobi polynomials as well as the associated Hasse-Witt invariant.  The techniques used in this paper are not specific to the Jacobi polynomials and are widely applicable. 
\end{abstract}

\maketitle

\section{Introduction} \label{Introduction}

\subsection{Motivation}

Let $A$ and $B$ be finite groups.  The group extension problem for the ordered pair $(A,B)$ is to determine all possible groups $G$ such that $A \vartriangleleft G$ and $G/A \simeq B$. The possible group extensions are well known to be classified by certain cohomology groups, which may be difficult to compute based on the properties of $A$ and $B$ (abelian or not, coprime order, etc.).   

The group extension problem has a natural application to Galois theory.  If $L/K$ is a Galois extension of fields with Galois group $B$, one can ask for the field-theoretic conditions that allow $L$ to be embedded in a field $M$ such that $\Gal(M/K) \simeq G$ and $\Gal(M/L) \simeq A$.

\begin{exm}
Let $C_n$ denote the cyclic group of order $n$.  Then $\Q(\sqrt{d})$, which has Galois group $C_2$ over $\Q$, can be embedded in a Galois extension $L/\Q$ with Galois group $C_4$ if and only if $d$ is the sum of two nonzero rational squares \cite[Thm.~1.2.4]{serre}.
\end{exm}

In \cite{feit}, Feit considered the cases where $A = C_2$ and $B = A_5$ or $A_7$ and solved the Galois embedding problem for the nontrivial central extensions $G = \widetilde{A}_5$ or $\widetilde{A}_7$, respectively.  In fact, Feit gave a parameterized family of polynomials with Galois group $A_5$ (resp.~$A_7$) solving the Galois embedding problem for $\widetilde{A}_5$ (resp.~$\widetilde{A}_7$).   In \cite{hajir}, Hajir extended the construction of Feit and showed that $\widetilde{A}_n$ is Galois over  $\Q$ whenever $n \equiv 1 \pmod{8}$.  We briefly recall their methods now.

\subsection{The Hasse-Witt Invariant} For complete details of this construction, see \cite[\S 2,3]{feit} and \cite{serre}.  Let $f \in \Q[x]$ be an irreducible polynomial of degree $n$ and set $E = \Q[x]/(f(x))$.  Fix an algebraic closure $\overline{\Q}$ of $\Q$ and let $\theta_1,\dots,\theta_n \in \overline{\Q}$ denote the roots of $f$.   Define the sequence of \textbf{power-sums} $\left\{s_k \right\}_{k=0}^\infty$ of the $\theta_j$  by 
\[
s_k = \sum_{j=1}^n \theta_j^k, 
\]
and let $\Delta_{t}$ be the principal $t \times t$ minor of the array $(s_{i+j})_{i,j \geq 0}$:
\begin{align}  \label{feit_minor}
\Delta_{t} \ddef \det \left(s_{i+j}\right)_{0 \leq i,j \leq t-1}.
\end{align}

\begin{rmk} \label{disc_rem}
It is well known (see, \emph{e.g.}, \cite[Thm.~2.2]{feit}) that $\Delta_n = \disc (f)$.
\end{rmk}

The linear transformation $Q_f:E \to E$ defined by $\gamma \mapsto \tr(\gamma^2)$ is a nondegenerate quadratic form on $E$.  By \cite[Thm.~2.2]{feit}, $Q_f$ is equivalent to the diagonal form
\begin{align} \label{diag}
\sum_{t=0}^{n-1} D_tx_t^2,
\end{align}
where  $D_t = \Delta_{t+1}/\Delta_{t}$ for $0 \leq t \leq n-1$.  

Let $p$ be a prime number and let $(~,~)_p: \Q_p^\times \times \Q_p^\times \to \lbrace \pm 1 \rbrace$ be the associated  Hilbert symbol (for properties of the Hilbert symbol, see \cite[\S 3]{feit}).  With the $D_t$ as above, define
\[
\epsilon_p(Q) = \prod_{0 \leq i < j \leq n-1} (D_i,D_j)_p.
\]
Following the notation of \cite[\S3]{feit}, we will write $\epsilon_p(E) = \epsilon_p(f(x)) = \epsilon_p(Q_f)$.  Then $\epsilon_p(E)$ is known as the \textbf{Hasse-Witt invariant} of the quadratic form $Q_f$, and its utility in Galois embedding problems is exemplified by the following theorem of Serre (appearing as Theorem 3.2 of  \cite{feit} and developed in generality in  \cite[\S2]{serre}).

\begin{thm} \label{serre_thm}
Let $E$ be a finite extension of $\Q$ and let $\widehat{E}$ denote its Galois closure in some algebraic closure.  Assume that $\widehat{E}$ has square discriminant.  Thus $G = \Gal(\widehat{E}/\Q) \subseteq A_n$, where $n = [E:\Q]$.  Suppose $\widetilde{G}$ is a nonsplit extension of $G$.  Then the following are equivalent.
\begin{enumerate}
\item There exists a quadratic extension field $M$ of $\widehat{E}$ which is a Galois extension of $\Q$ with $\Gal(M/\Q) \simeq \widetilde{G}$.
\item $\epsilon_p(E) =1$ for all finite primes $p$ and for the Archimedean prime $p = \infty$.
\end{enumerate}
\end{thm}

In this paper we first recall Feit's calculation of the explicit form of the diagonalization of $Q_f$ for a one-parameter family of polynomials (Generalized Laguerre) and then introduce a new method to do the same for a two variable family (Jacobi), giving us the following result.

\begin{thm} \label{explicit_diag}
For the Jacobi polynomial $f(x) = \mathscr{J}^{(\alpha, \beta)}_n(x)$, the diagonalization $Q_f$ of the trace form of $f$ is given by $\sum_{t=0}^{n-1} D_t x_t^2$, where
\[
D_t = \frac{n! \prod_{j=n-t+1}^n (\alpha + j)(\beta+j)}{(n-t-1)! (\alpha + \beta + 2n-2t+1) \prod_{j=2n-2t+2}^{2n}(a+b+j)}.
\]
\end{thm}

At the end of the paper we also apply our method to give an alternative derivation of Feit's theorem.   We now describe these polynomials in detail.

\subsection{The Generalized Laguerre Polynomials} One of the barriers to applying Theorem \ref{serre_thm} in practice is the calculation of the $D_i$ in the trace form above.  This amounts to determining an explicit formula for all $t \times t$ principal minors of the $n \times n$ power-sum matrix of the roots of the defining polynomial of the root field $E$.  In \cite{feit}, Feit considered the one-parameter family $\lbrace F_n^{(\lambda)}(x)\rbrace_{n=0}^\infty$ of \textbf{Generalized Laguerre Polynomials}, defined by
\[
F_n^{(\lambda)}(x) = \sum_{j=0}^n (-1)^{n-j} \binom{n}{j} \left(\prod_{i=j+1}^n \lambda + i \right)x^j.
\]
Using properties specific to the $F_n^{(\lambda)}(x)$, such as their three-term recursion relations and certain congruences, Feit determined $\Delta_{t}$ explicitly for all $t$ and, as a consequence, gave specific instances where $\epsilon_p(E) = 1$.   This was the major computational component to \cite{feit}, and is how he produced infinitely many number fields $\widehat{E}$ (depending on $\lambda$) with Galois group $A_5$ or $A_7$, that embed into $\widetilde{A}_5$- or $\widetilde{A}_7$-extensions, respectively.  In \cite{hajir}, Hajir generalized the work of Feit and used the polynomials $F_n^{(-2-n)}(x)$ to obtain $\widetilde{A}_n$-Galois extensions of $\Q$ when $n \equiv 1 \pmod{8}$. 

\subsection{Current Work} In this paper we expand the work of Feit and Hajir to the \textbf{Jacobi Polynomials}, a two-parameter family of orthogonal polynomials generalizing the Generalized Laguerre Polynomials.  The additional parameter of the Jacobi polynomials will mean that for each $n \geq 2$ there may be infinitely many solutions to the group extension problem for $\widetilde{A}_n$, thereby generalizing the results of Feit and Hajir. The main difficulty is that the explicit techniques of Feit to compute the $\Delta_t$ do not apply to the Jacobi Polynomials.  

By contrast, we determine the $\Delta_{t}$ for the Jacobi polynomials by recasting the problem in terms of Hankel determinants and then applying known combinatorial techniques to evaluate the minors.   As a consequence of our work on the Jacobi Polynomials we will get a new proof of \cite[Thm.~5.1]{feit} on the formula for $\Delta_t$ for the Generalized Laguerre Polynomials.  Our method to compute $\Delta_t$ is applicable to any sequence of polynomials indexed by degree and makes no use of orthogonality.  

Recall that the Jacobi Polynomials $P_n^{(\alpha,\beta)}(x)$ are a two-parameter family of orthogonal polynomials defined explicitly as
\[
P_n^{(\alpha,\beta)}(x) = \sum_{j=0}^n \binom{n+\alpha}{n-j}\binom{n+\beta}{j} \left(\frac{x-1}{2}\right)^j\left( \frac{x+1}{2} \right)^{n-j}.
\]
The $P_n^{(\alpha,\beta)}(x)$ are orthogonal on $[-1,1]$ with respect to the weight function $w(x) = (1-x)^\alpha(1+x)^\beta$.

\begin{rmk}
If we set $L_n^{(\lambda)}(x) = F_n^{(\lambda)}(x)/n!$, then the relationship between the Jacobi and Generalized Laguerre polynomials is given by \cite[(5.3.4)]{szego}
\[
L_n^{(\alpha)}(x) = \lim_{\beta \to \infty}P_n^{(\alpha,\beta)}(1-2x/\beta).
\]
Other changes of variables and specializations show that the Hermite, Legendre, Ultraspherical, and Chebychev polynomials (of both kinds) can be viewed as subfamilies of the Jacobi polynomials (see \cite[\S 2.4]{szego} for the explicit relationships).
\end{rmk}

As written, the $P_n^{(\alpha,\beta)}(x)$ are not monic and have complicated expressions for their coefficients.  We will make a series of convenient transformations for studying their power-sum sequence. First, we define the shifted polynomials $J_n^{(\alpha,\beta)}(x)$ to be 
\[
J_n^{(\alpha,\beta)}(x) \ddef P_n^{(\alpha,\beta)}(2x+1),
\]
which are given explicitly \cite[4.21.2]{szego} by
\[
J_n^{(\alpha,\beta)}(x) = \sum_{j=0}^n \binom{n+\alpha}{n-j}\binom{n+\alpha + \beta +j}{j} x^j.
\]
Then we define the monic polynomials $\mathscr{J}_n^{(\alpha,\beta)}(x)$ by 
\begin{align} \label{jac_def}
\mathscr{J}_n^{(\alpha,\beta)}(x) &\ddef (-1)^n\binom{2n+\alpha + \beta}{n}^{-1} J_n^{(\alpha,\beta)}(-x) \\
&= \sum_{j=0}^n \binom{n}{j} \left( \prod_{k=j+1}^n \frac{\alpha + k}{n+\alpha + \beta + k} \right) (-1)^{n-j}x^j.
\end{align}
We define the following quantities as well:
\begin{align}
A_{t} &= \prod_{j=n-t+1}^n j^{j-(n-t)}, \label{ADEF} \\
B_{t} &= \prod_{j=n-t+2}^n \left[(\alpha + j)(\beta+j)\right]^{j-(n-t+1)} \left(\alpha + \beta + j\right)^{j-(n-t+2)}, \label{BDEF} \\
C_{t} &= \prod_{j=2n-2t+2}^{2n} (\alpha + \beta + j)^{j-(2n-2t+2)}. \label{CDEF}
\end{align}
With this notation in place, we can now state the main results of this paper.

\begin{thm} \label{main_thm}
Fix $n \geq 2$.  For the polynomial  $\mathscr{J}_n^{(\alpha,\beta)}(x)$, we have $\Delta_{t} = A_{t}B_{t}/C_{t}$, with $A_t$, $B_t$, and $C_t$  as defined in Equations (\ref{ADEF}), (\ref{BDEF}), and (\ref{CDEF}).
\end{thm}

In \cite[Thm.~3.1]{feit}, Feit proved the general formula
\begin{align} \label{HW_general}
\epsilon_p(E) = \left\{ \prod_{t=1}^n \left( \Delta_{t-1},\Delta_{t} \right)_p \right\} \left( -1,\prod_{t=1}^n \Delta_{t} \right)_p,
\end{align}
where $(~,~)_p$ denotes the $p$-adic Hilbert symbol.  By applying Theorem \ref{main_thm} to Equation (\ref{HW_general}) we will obtain in Section \ref{HW_section} an explicit expression for the Hasse-Witt invariant of the root field of $\mathscr{J}_n^{(\alpha,\beta)}(x)$.  We then show how to recover \cite[Thm.~5.1]{feit} immediately from Theorem \ref{main_thm} and give several concrete examples pertaining to the Jacobi Polynomials.   

\section{Hankel Matrices}

\subsection{Determinants} A \textbf{Hankel matrix} is one that is constant along anti-diagonals:
\[
\begin{pmatrix} 
a & b & c & \cdots \\
b & c & d & \cdots \\
c & d & e &  \cdots \\
\vdots & \vdots & \vdots & \ddots
\end{pmatrix}.
\]
If $\lbrace \mu_k \rbrace_{k\geq 0}$ is a sequence, then we can define the associated $t \times t$ Hankel matrix via 
\[
\left( \mu_{i+j} \right)_{0 \leq i,j \leq t-1}.
\]
Given such a Hankel matrix, one is led to consider its determinant 
\begin{align} \label{general_det}
\det \left(\mu_{i+j}\right)_{0 \leq i,j \leq t-1}.
\end{align} 
In certain circumstances, if the sequence $\lbrace \mu_k \rbrace_{k \geq 0}$ is endowed with some extra structure, there are techniques for evaluating the determinant in (\ref{general_det}). We recall one of those circumstances now.

Let $\lbrace \mu_k \rbrace_{k \geq 0}$ be a sequence with generating function
\[
G(x) =  \sum_{k=0}^\infty \mu_kx^k.
\]
If we write $G(x)$ as a continued fraction of the form
\begin{align} \label{cont_frac}
G(x) = \sum_{k=0}^\infty \mu_kx^k  = \frac{\mu_0}{\displaystyle 1+  \frac{a_1x}{1+ \displaystyle \frac{a_2x}{1 + \cdots}}},
\end{align}
then by \cite[Thm.~30]{krattenhaler}, we have
\begin{align} \label{kratt_form}
\det \left( \mu_{i+j} \right)_{0 \leq i,j \leq t-1} = \mu_0^t \left( a_1a_2 \right)^{t-1}\left(a_3a_4 \right)^{t-2} \cdots \left(a_{2t-3}a_{2t-2}\right).
\end{align}
The minor $\Delta_{t}$ introduced above is a Hankel determinant, hence one approach to proving Theorem \ref{main_thm} is to determine the continued fraction expansion as in Equation (\ref{cont_frac}) for the generating function of the power-sums.

\begin{rmk}
In \cite{feit}, Feit considered the determinant 
\[
\det \left(s_{i+j-2} \right)_{1 \leq i,j \leq t}.
\]
In order to have Feit's indexing agree with the indexing above, we set $i \mapsto i-1$ and $j \mapsto j-1$ and start the indexing at 0.
\end{rmk}

\subsection{Generating Functions}

Let $F \in \Q[x]$ be a polynomial of degree $n \geq 1$ with roots $\theta_1,\dots,\theta_n \in \overline{\Q}$.  Let $G(z) = \sum_{i=0}^\infty s_iz^i$ be the generating function for the power-sums of the roots.  For any polynomial $f$, let $\widehat{f}$ denote the reciprocal polynomial.

\begin{prop} \label{G_deriv_prop}
With all notation as above, we have
\[
G(z) = \frac{\widehat{F'(z)}}{\widehat{F(z)}}.
\]
\end{prop}

\begin{proof}
Fix an index $j \in \lbrace 1,\dots,n \rbrace$.  Then
\[
\frac{1}{1-z\theta_j} = \sum_{i=0}^\infty \theta_j^iz^i.
\]
Therefore, 
\[
G(z) = \sum_{i=0}^\infty s_iz^i = \sum_{i=0}^\infty \sum_{j=1}^n \theta_j^i z^i = \sum_{j=1}^n \sum_{i=0}^\infty \theta_j^iz^i= \sum_{j=1}^n \frac{1}{1-z\theta_j} =  \frac{\widehat{F'(z)}}{\widehat{F(z)}}.
\]
\end{proof}

Our strategy for determining $\Delta_t$ will be to write $G(z) = {\widehat{F'(z)}}/{\widehat{F(z)}}$ in the form of Equation (\ref{cont_frac}), thereby determining the coefficients $a_i$. Then
\begin{align} \label{det_k}
\Delta_{t} = s_0^t \left( a_1a_2 \right)^{t-1}\left(a_3a_4 \right)^{t-2} \cdots \left(a_{2t-3}a_{2t-2}\right),
\end{align}
with $s_0=n$.

\section{Proof of Theorem \ref{main_thm}}

For ease of exposition, let 
\begin{align} \label{gamma_defn}
\gamma_j(n)^{(\alpha,\beta)} = \binom{n}{j} \left( \prod_{k=j+1}^n \frac{\alpha + k}{n+\alpha + \beta + k} \right) (-1)^{n-j},
\end{align}
so that $\mathscr{J}_n^{(\alpha,\beta)}(x) = \sum_{j=0}^n\gamma_j(n)^{(\alpha,\beta)}x^j$. We start by writing $G(z)$ in the form of Proposition \ref{G_deriv_prop}.

\begin{prop}\label{gen_function}
Let $n \geq 1$.  The generating function $G(z)$ for the power-sum polynomials in the roots of $\mathscr{J}_n^{(\alpha,\beta)}(x)$ is given by
\begin{align}  \label{F_eq_original}
G(z) =\frac{\sum_{j=0}^{n-1} \gamma_j^{(\alpha+1,\beta+1)}(n-1)z^{n-1-j}}{\sum_{j=0}^n \gamma_j^{(\alpha,\beta)}(n)z^{n-j}}.
\end{align}
\end{prop}

\begin{proof}
Fix $n \geq 1$.  By Proposition \ref{G_deriv_prop}, 
\[
G(z) = \frac{ \widehat{ \mathscr{J}_n^{(\alpha,\beta)}(z)'}}{\widehat{ \mathscr{J}_n^{(\alpha,\beta)}(x)}}.
\]
It is well known that
\[
\frac{{\rm d}}{{\rm d}z} P_n^{(\alpha,\beta)}(z) = \frac{\Gamma(\alpha+\beta+n+2)}{2\Gamma(\alpha+\beta+n+1)} P_{n-1}^{(\alpha+1,\beta+1)}(z),
\]
which gives us
\[
\frac{{\rm d}}{{\rm d}z} \mathscr{J}_n^{(\alpha,\beta)}(z) = n \mathscr{J}_{n-1}^{(\alpha + 1,\beta+1)}(z).
\]
Equipped with this derivative formula, the claimed expression for $G(z)$ follows by direct calculation. 
\end{proof}

We now turn to the proof of Theorem \ref{main_thm}.  To do this we will determine the $a_i$ of Equation (\ref{det_k}); it will then follow by algebraic manipulation that Equation (\ref{det_k}) and the one of Theorem \ref{main_thm} are the same.

Fix $n \geq 1$ and set the following notation. Let $c_i=i\gamma_i^{(\alpha,\beta)}(n)$ and $d_i=\gamma_i^{(\alpha,\beta)}(n)$ for all $0\le i\le n$, so that $c_i = id_i$. We define $f_0,\ldots,f_{2n-1}:\{0,\dots,n\} \to \R$ as follows. Let 
\begin{align}
f_0(m)&=\begin{cases} c_m,&1\le m\le n\\
0,&m=0, \end{cases} \label{f_0_eq}
\\
f_1(m)&=\begin{cases} d_{m-1}-\frac{f_0(m-1)}{f_0(n)},&1\le m\le n\\
0,&m=0,\hspace{0.5in}\text{and} \end{cases} \label{f_1_eq}
\\
f_i(m)&=\begin{cases}
\frac{f_{i-2}(m-1)}{f_{i-2}(n)}-\frac{f_{i-1}(m-1)}{f_{i-1}(n)},&\frac{i}{2}< m\le n\\
0,&0\le m\le\frac{i}{2}
\end{cases} \label{f_i_eq} \end{align}
for $2\le i\le 2n-1$. Let $a_i=f_i(n)$ for  $0\le i\le 2n-1$. We also note that for  $1\le m\le n$ we can simplify Equation (\ref{f_1_eq}) to get
\begin{align} \label{f_1_simplified}
    f_1(m)&=\gamma_{m-1}^{(\alpha,\beta)}(n)-\frac{(m-1)\gamma_{m-1}^{(\alpha,\beta)}(n)}{n} =\gamma_{m-1}^{(\alpha,\beta)}(n)\cdot\frac{n-m+1}{n}.
\end{align}

\begin{lem} \label{F_eq_cd}
We have \[G(z)=\frac{\sum_{j=1}^{n}c_jz^{n-j}}{\sum_{j=0}^nd_jz^{n-j}}.\]
\end{lem}

\begin{proof}
Using Equation (\ref{gamma_defn}) we have
\begin{align*}
    \gamma_j^{(\alpha+1,\beta+1)}(n-1)=\frac{j+1}{n}\gamma_{j+1}^{(\alpha,\beta)}(n),
\end{align*}
which we substitute into Equation (\ref{F_eq_original}) to get
\begin{align*}
    G(z)=\frac{\sum_{j=1}^{n}j\gamma_j^{(\alpha,\beta)}(n)z^{n-j}}{\sum_{j=0}^n\gamma_j^{(\alpha,\beta)}(n)z^{n-j}}.
\end{align*}
Applying the definitions of $c_i$ and $d_i$ completes the proof.
\end{proof}

\begin{prop} \label{cont_frac_prop}
The generating function $G(z)$ of Proposition \ref{gen_function} is given by  \[G(z)=\dfrac{a_0}{1+\dfrac{a_1z}{1+\dfrac{a_2z}{\dfrac{\vdots}{1+\dfrac{a_{2n-2}z}{1+a_{2n-1}z}}}}},\]
with $a_i=f_i(n)$ defined as above.
\end{prop}

\begin{proof}
We will show that 
\begin{align*}
G(z)=\dfrac{a_0}{1+\dfrac{a_1z}{1+\dfrac{a_2z}{\dfrac{\vdots}{1+\dfrac{a_{2q-1}z}{\dfrac{\sum_{j=q}^n\frac{f_{2q-2}(j)}{f_{2q-2}(j)}z^{n-j}}{\sum_{j=q}^n\frac{f_{2q-1}(j)}{f_{2q-1}(n)}z^{n-j}}}}}}}
\end{align*}
for all $1\le q\le n$; setting $q=n$ will complete the proof. We will induct on $q$.

By Lemma \ref{F_eq_cd} we have
\begin{align*}
    G(z)=\frac{\sum_{j=1}^{n}c_jz^{n-j}}{\sum_{j=0}^nd_jz^{n-j}}.
\end{align*}
Note that $c_j=f_0(j)$ for all $1\le j\le n$. We also have $c_n=f_0(n)=a_0$ and $d_n=\gamma_n^{(\alpha,\beta)}(n)=1$. In particular, $a_0 = n$. Hence
\begin{align*}
    G(z)&=\frac{\bigg(\sum_{j=1}^{n-1}f_0(j)z^{n-j}\bigg)+a_0}{\bigg(\sum_{j=0}^{n-1}d_jz^{n-j}\bigg)+1} =\dfrac{a_0}{\dfrac{\bigg(\sum_{j=0}^{n-1}d_jz^{n-j}\bigg)+1}{\bigg(\sum_{j=1}^{n-1}\frac{f_0(j)}{f_0(n)}z^{n-j}\bigg)+1}}.
\end{align*}
Observe that 
\begin{align*}
{\dfrac{\bigg(\sum_{j=0}^{n-1}d_jz^{n-j}\bigg)+1}{\bigg(\sum_{j=1}^{n-1}\frac{f_0(j)}{f_0(n)}z^{n-j}\bigg)+1}} &= {1+\dfrac{\bigg(\sum_{j=0}^{n-1}(d_j-\frac{f_0(j)}{f_0(n)})z^{n-j}\bigg)}{\bigg(\sum_{j=1}^{n-1}\frac{f_0(j)}{f_0(n)}z^{n-j}\bigg)+1}} \\
&={1+\dfrac{z\bigg(\sum_{j=1}^nf_1(j)z^{n-j}\bigg)}{\bigg(\sum_{j=1}^{n-1}\frac{f_0(j)}{f_0(n)}z^{n-j}\bigg)+1}}.
\end{align*}
Thus,
\begin{align*}    
     G(z)&=\dfrac{a_0}{1+\dfrac{z\bigg(\sum_{j=1}^nf_1(j)z^{n-j}\bigg)}{\bigg(\sum_{j=1}^{n-1}\frac{f_0(j)}{f_0(n)}z^{n-j}\bigg)+1}}
    =\dfrac{a_0}{1+\dfrac{a_1z}{\dfrac{\sum_{j=1}^n\frac{f_0(j)}{f_0(n)}z^{n-j}}{\sum_{j=1}^n\frac{f_1(j)}{f_1(n)}z^{n-j}}}},
\end{align*}
which completes the base case. Fix $r\in\{1,\ldots,n-1\}$. Suppose 
\begin{align*}
G(z)=\dfrac{a_0}{1+\dfrac{a_1z}{1+\dfrac{a_2z}{\dfrac{\vdots}{1+\dfrac{a_{2r-1}z}{\dfrac{\sum_{j=r}^n\frac{f_{2r-2}(j)}{f_{2r-2}(n)}z^{n-j}}{\sum_{j=r}^n\frac{f_{2r-1}(j)}{f_{2r-1}(n)}z^{n-j}}}}}}}.
\end{align*}
Then
\begin{align*}
{1+\dfrac{a_{2r-1}z}{\dfrac{\sum_{j=r}^n\frac{f_{2r-2}(j)}{f_{2r-2}(n)}z^{n-j}}{\sum_{j=r}^n\frac{f_{2r-1}(j)}{f_{2r-1}(n)}z^{n-j}}}} &=
{1+\dfrac{a_{2r-1}z}{1+\dfrac{\sum_{j=r}^{n-1}\bigg(\frac{f_{2r-2}(j)}{f_{2r-2}(n)}-\frac{f_{2r-1}(j)}{f_{2r-1}(n)}\bigg)z^{n-j}}{\sum_{j=r}^n\frac{f_{2r-1}(j)}{f_{2r-1}(n)}z^{n-j}}}} \\
&={1+\dfrac{a_{2r-1}z}{1+\dfrac{a_{2r}z}{\dfrac{\sum_{j=r}^n\frac{f_{2r-1}(j)}{f_{2r-1}(n)}z^{n-j}}{\sum_{j=r+1}^{n}\frac{f_{2r}(j)}{f_{2r}(n)}z^{n-j}}}}} \\
&={1+\dfrac{a_{2r-1}z}{1+\dfrac{a_{2r}z}{1+\dfrac{\sum_{j=r}^{n-1}\big(\frac{f_{2r-1}(j)}{f_{2r-1}(n)}-\frac{f_{2r}(j)}{f_{2r}(n)}\big)z^{n-j}}{\sum_{j=r+1}^{n}\frac{f_{2r}(j)}{f_{2r}(n)}z^{n-j}}}}}\\
&={1+\dfrac{a_{2r-1}z}{1+\dfrac{a_{2r}z}{1+\dfrac{a_{2r+1}z}{\dfrac{\sum_{j=r+1}^{n}\frac{f_{2r}(j)}{f_{2r}(n)}z^{n-j}}{\sum_{j=r+1}^{n}\frac{f_{2r+1}(j)}{f_{2r+1}(n)}z^{n-j}}}}}},
\end{align*}
which completes the proof.
\end{proof}

Next we determine alternate expressions for the $f_i$ that are more convenient for explicit calculation.

\begin{prop} \label{f(n-l)_prop}
For all $0\le i\le n-1$ and all $1\le \ell<n-\frac{i}{2}$, we have
\begin{align*}
    f_{2i}(n-\ell)&=f_{2i}(n-\ell+1)\cdot\frac{-(n-i-\ell)(\alpha+n-i-\ell+1)}{\ell(\alpha+\beta+2n-2i-\ell+1)}, \text{ and } \\
    f_{2i+1}(n-\ell)&=f_{2i+1}(n-\ell+1)\cdot\frac{-(n-i-\ell)(\alpha+n-i-\ell)}{\ell(\alpha+\beta+2n-2i-\ell)}.
\end{align*}    
\end{prop}

\begin{proof}
We will proceed by induction on $i$, beginning with the cases $i \in \lbrace 0,1\rbrace$ separately.  Applying  the definition of the $\gamma_{m}^{(\alpha,\beta)}$ we have
\begin{align*}
    f_0(n-\ell)&=(n-\ell)\gamma_{n-\ell}^{(\alpha,\beta)}(n)\\
    &=(n-\ell)\binom{n}{n-\ell}\bigg(\prod_{k=n-\ell+1}^n\frac{\alpha+k}{\alpha+\beta+n+k}\bigg)(-1)^\ell \\
    &=\frac{-(n-\ell)(n-\ell+1)}{\ell}\binom{n}{n-\ell+1}\bigg(\prod_{k=n-\ell+2}^n\frac{\alpha+k}{\alpha+\beta+n+k}\bigg)\bigg(\frac{\alpha+n-\ell+1}{\alpha+\beta+2n-\ell+1}\bigg)(-1)^{\ell-1}\\
    &=(n-\ell+1)\gamma_{n-\ell+1}^{(\alpha,\beta)}(n)\cdot\frac{-(n-\ell)(\alpha+n-\ell+1)}{\ell(\alpha+\beta+2n-\ell+1)}\\
    &=f_0(n-\ell+1)\cdot\frac{-(n-\ell)(\alpha+n-\ell+1)}{\ell(\alpha+\beta+2n-\ell+1)}, \text{ and} \\
    f_1(n-\ell)&=\gamma_{n-\ell-1}^{(\alpha,\beta)}(n)\frac{\ell+1}{n} \qquad \text{(by Equation (\ref{f_1_simplified}))}\\
    &=\binom{n}{n-\ell-1}\bigg(\prod_{k=n-\ell}^n\frac{\alpha+k}{\alpha+\beta+n+k}\bigg)(-1)^{\ell+1}\cdot\frac{\ell+1}{n}\\
    &=\binom{n}{n-\ell}\bigg(\prod_{k=n-\ell+1}^n\frac{\alpha+k}{\alpha+\beta+n+k}\bigg)(-1)^{\ell}\cdot\frac{-(n-\ell)(\alpha+n-\ell)}{n(\alpha+\beta+2n-\ell)}\\
    &=\gamma_{n-\ell}^{(\alpha,\beta)}(n)\cdot\frac{\ell}{n}\cdot\frac{-(n-\ell)(\alpha+n-\ell)}{\ell(\alpha+\beta+2n-\ell)}\\
    &=f_1(n-\ell+1)\cdot\frac{-(n-\ell)(\alpha+n-\ell)}{\ell(\alpha+\beta+2n-\ell)},
\end{align*}
which completes the base cases. Now fix $s\in\{0,\ldots,n-2\}$ and suppose 
\begin{align} \label{f_2s_induct_hyp}
    f_{2s}(n-\ell)&=f_{2s}(n-\ell+1)\cdot\frac{-(n-s-\ell)(\alpha+n-s-\ell+1)}{\ell(\alpha+\beta+2n-2s-\ell+1)},\text{ and} \\
 \label{f_2s+1_induct_hyp}
    f_{2s+1}(n-\ell)&=f_{2s+1}(n-\ell+1)\cdot\frac{-(n-s-\ell)(\alpha+n-s-\ell)}{\ell(\alpha+\beta+2n-2s-\ell)},
\end{align}
for all $1\le \ell<n-\frac{s}{2}$. Using Equations (\ref{f_i_eq}), (\ref{f_2s_induct_hyp}), and (\ref{f_2s+1_induct_hyp}) we have
\begin{align}
f_{2s+2}(n-\ell)&=\frac{f_{2s}(n-\ell-1)}{f_{2s}(n)}-\frac{f_{2s+1}(n-\ell-1)}{f_{2s+1}(n)}\notag \\
&=\bigg(\prod_{k=1}^{\ell+1}\frac{-(n-s-k)(\alpha+n-s-k+1)}{k(\alpha+\beta+2n-2s-k+1)}\bigg)-\bigg(\prod_{k=1}^{\ell+1}\frac{-(n-s-k)(\alpha+n-s-k)}{k(\alpha+\beta+2n-2s-k)}\bigg)\notag \\
&=\bigg(\prod_{k=1}^{\ell+1}\frac{-(n-s-k)}{k}\bigg)\bigg(\prod_{k=1}^\ell\frac{\alpha+n-s-k}{\alpha+\beta+2n-2s-k}\bigg) \notag 
\\&\hspace{1in}\cdot\bigg(\frac{\alpha+n-s}{\alpha+\beta+2n-2s}-\frac{\alpha+n-s-\ell-1}{\alpha+\beta+2n-2s-\ell-1}\bigg)\notag  \\
&=\bigg(\prod_{k=1}^{\ell+1}\frac{-(n-s-k)}{k}\bigg)\bigg(\prod_{k=1}^\ell\frac{\alpha+n-s-k}{\alpha+\beta+2n-2s-k}\bigg) \notag 
\\&\hspace{1in}\cdot\bigg(\frac{(\ell+1)(\beta+n-s)}{(\alpha+\beta+2n-2s)(\alpha+\beta+2n-2s-\ell-1)}\bigg)\notag \\
&=\bigg(\prod_{k=1}^{\ell}\frac{-(n-s-k)}{k}\bigg)\bigg(\prod_{k=1}^{\ell-1}\frac{\alpha+n-s-k}{\alpha+\beta+2n-2s-k}\bigg) \notag 
\\&\hspace{0.5in}\cdot\bigg(\frac{-(\beta+n-s)(n-s-\ell-1)(\alpha+n-s-\ell)}{(\alpha+\beta+2n-2s)(\alpha+\beta+2n-2s-\ell-1)(\alpha+\beta+2n-2s-\ell)}\bigg)\notag \\
&=\bigg(\prod_{k=1}^{\ell}\frac{-(n-s-k)}{k}\bigg)\bigg(\prod_{k=1}^{\ell-1}\frac{\alpha+n-s-k}{\alpha+\beta+2n-2s-k}\bigg) \notag 
\\&\hspace{0.5in}\cdot\bigg(\frac{\alpha+n-s}{\alpha+\beta+2n-2s}-\frac{\alpha+n-s-\ell}{\alpha+\beta+2n-2s-\ell}\bigg)\bigg(\frac{-(n-s-\ell-1)(\alpha+n-s-\ell)}{\ell(\alpha+\beta+2n-2s-\ell-1)}\bigg)\notag \\
&=\bigg[\bigg(\prod_{k=1}^\ell\frac{-(n-s-k)(\alpha+n-s-k+1)}{k(\alpha+\beta+2n-2s-k+1)}\bigg)-\bigg(\prod_{k=1}^\ell\frac{-(n-s-k)(\alpha+n-s-k)}{k(\alpha+\beta+2n-2s-k)}\bigg)\bigg] \notag 
\\&\hspace{0.5in}\cdot\bigg(\frac{-(n-s-\ell-1)(\alpha+n-s-\ell)}{\ell(\alpha+\beta+2n-2s-\ell-1)}\bigg)\notag \\
&=\bigg(\frac{f_{2s}(n-\ell)}{f_{2s}(n)}-\frac{f_{2s+1}(n-\ell)}{f_{2s+1}(n)}\bigg)\bigg(\frac{-(n-s-\ell-1)(\alpha+n-s-\ell)}{\ell(\alpha+\beta+2n-2s-\ell-1)}\bigg). \notag
\end{align}
This shows that 
\begin{align}
f_{2s+2}(n-\ell)&=f_{2s+2}(n-\ell+1)\cdot\frac{-(n-s-\ell-1)(\alpha+n-s-\ell)}{\ell(\alpha+\beta+2n-2s-\ell-1)}. \label{f_2s+2_induct_step}
\end{align}
A similar argument using Equations (\ref{f_i_eq}), (\ref{f_2s+1_induct_hyp}), and (\ref{f_2s+2_induct_step}) shows that
\begin{align}
f_{2s+3}(n-\ell)
&=f_{2s+3}(n-\ell+1)\cdot\frac{-(n-s-\ell-1)(\alpha+n-s-\ell-1)}{\ell(\alpha+\beta+2n-2s-\ell-2)}. \label{f_2s+3_induct_step}
\end{align}
This completes the proof.
\end{proof}

\begin{thm} \label{elis_version}
Fix $n \geq 1$.  Let $\{s_k\}_{k\geq 0}$ be the sequence of power-sums for the roots of  $\mathscr{J}_n^{(\alpha,\beta)}(x)$.  Then 
\begin{gather*}
    \det(s_{i+j})_{0\le i,j\le t-1}={s_0}^t(a_1a_2)^{t-1}
   (a_3a_4)^{t-2}\cdots(a_{2n-3}a_{2n-2}),
\end{gather*}
where $s_0 = a_0=n$, ${a_1=\frac{-(\alpha+n)}{(\alpha+\beta+2n)}}$, 
\begin{align*}
    a_{2i}&=\frac{-(n-i)(\beta+n-i+1)}{(\alpha+\beta+2n-2i+2)(\alpha+\beta+2n-2i+1)}, \\
    a_{2i+1}&=\frac{-(\alpha+n-i)(\alpha+\beta+n-i+1)}{(\alpha+\beta+2n-2i+1)(\alpha+\beta+2n-2i)},\text{ and}\\
\end{align*}
for  $1\le i\le n-1$.
\end{thm}

\begin{proof}
Recall from Equations (\ref{cont_frac}) and (\ref{kratt_form}) that given the generating function $G(z)$ for any sequence $\lbrace \mu_k \rbrace_{k \geq 0}$ written as a continued fraction, the associated Hankel determinant is given by
\begin{align} \label{det_form_defn}
    \det(\mu_{i+j})_{0\le i,j\le t-1}={\mu_0}^t(a_1a_2)^{t-1}
   (a_3a_4)^{t-2}\cdots(a_{2n-3}a_{2n-2}).
\end{align}
Having proved Proposition \ref{cont_frac_prop}, it remains to provide explicit expressions for the coefficients $a_0,\ldots a_{2n-2}$.  We have already observed that $s_0 = a_0 = n$.  By  Equation (\ref{f_1_simplified}) we have
\begin{align*}
    a_1=f_1(n)=\frac{\gamma_{n-1}^{(\alpha,\beta)}(n)}{n}=\frac{-(\alpha+n)}{\alpha+\beta+2n}.
\end{align*}
Fix $p\in\{1,\ldots,n-1\}$. Using Equation (\ref{f_i_eq}) and Proposition \ref{f(n-l)_prop} we have
\begin{align*}
    a_{2p}&=f_{2p}(n)\\
    &=\frac{f_{2p-2}(n-1)}{f_{2p-2}(n)}-\frac{f_{2p-1}(n-1)}{f_{2p-1}(n)}\\
    &=\frac{-(n-p)(\alpha+n-p+1)}{\alpha+\beta+2n-2p+2}-\frac{-(n-p)(\alpha+n-p)}{\alpha+\beta+2n-2p+1}\\
    &=\frac{-(n-p)(\beta+n-p+1)}{(\alpha+\beta+2n-2p+2)(\alpha+\beta+2n-2p+1)},
\end{align*}
and
\begin{align*}
    a_{2p+1}&=f_{2p+1}(n)\\
    &=\frac{f_{2p-1}(n-1)}{f_{2p-1}(n)}-\frac{f_{2p}(n-1)}{f_{2p}(n)}\\
    &=\frac{-(n-p)(\alpha+n-p)}{\alpha+\beta+2n-2p+1}-\frac{-(n-p-1)(\alpha+n-p)}{\alpha+\beta+2n-2p}\\
    &=\frac{-(\alpha+n-p)(\alpha+\beta+n-p+1)}{(\alpha+\beta+2n-2p+1)(\alpha+\beta+2n-2p)}.
\end{align*}
\end{proof}

We now complete the proof of Theorem \ref{main_thm}.

\begin{proof}[Proof of Theorem \ref{main_thm}]
By Theorem \ref{elis_version}, we have
\[
\Delta_{t} = \det (s_{i+j})_{0 \leq i,j \leq t-1} = {s_0}^t(a_1a_2)^{t-1}(a_3a_4)^{t-2}\cdots(a_{2n-3}a_{2n-2}),
\]
where the $s_{k}$ are the $k$th-power-sums of the roots of $\mathscr{J}_n^{(\alpha,\beta)}(x)$ and the $a_i$ are given explicitly in the statement of the theorem.  It then follows by routine algebraic manipulation that 
\[
{s_0}^t(a_1a_2)^{t-1}(a_3a_4)^{t-2}\cdots(a_{2n-3}a_{2n-2}) = A_{t}B_{t}/C_{t},
\]
where $A_{t}$, $B_{t}$, and $C_{t}$ are given by Equations (\ref{ADEF}), (\ref{BDEF}), and (\ref{CDEF}), respectively.  
\end{proof}

\section{Hasse-Witt Invariant} \label{HW_section}

Fix a positive integer $n$.  Our goal is to give an explicit expression for the Hasse-Witt invariant $\epsilon_p(E)$ of the root field $E$ defined by $\mathscr{J}_n^{(\alpha,\beta)}(x)$ when it is irreducible over $\Q$ with square discriminant.  
Recall from Equation (\ref{HW_general}) that
\[
\epsilon_p(E) = \left\{ \prod_{t=1}^n \left( \Delta_{t-1},\Delta_{t} \right)_p \right\} \left( -1,\prod_{t=1}^n \Delta_{t} \right)_p.
\]
Let $\sim$ denote the equivalence relation on $\Q_p^\times$ defined by $a \sim b$ if and only $a/b$ is a square in $\Q_p^\times$.  Recall that the Hilbert symbol is symmetric, bi-multiplicative, and satisfies $(a,b)_p = 1$ if either $a$ or $b$ is a square.  Factoring $\Delta_{t}$ as in Theorem \ref{main_thm}, we have 
\begin{align*}
\left( -1,\prod_{t=1}^n \Delta_{t} \right)_p  &= \left( -1,\prod_{t=1}^n A_{t}B_{t}/C_{t} \right)_p \\
&=   \left( -1,\prod_{t=1}^n A_{t} \right)_p\left( -1,\prod_{t=1}^n B_{t} \right)_p\left( -1,\prod_{t=1}^n C_{t} \right)_p.
\end{align*}
Now we compute
\begin{align}
\prod_{t=1}^n A_{t} &=  \prod_{t =1}^n t^{t(t+1)/2} \\
\prod_{t=1}^n B_{t} &=  \prod_{t=1}^n \left((\alpha +t)(\beta+t)\right)^{t(t-1)/2} \left( \alpha + \beta + t \right)^{(t-1)(t-2)/2}\\
\prod_{t=1}^n C_{t} &=  \prod_{t=1}^n (\alpha + \beta + 2t)^{t(t-1)} (\alpha + \beta + 2k-1)^{(t-1)^2}.
\end{align}
It is routine to check that
\begin{align}
\prod_{t=1}^n A_{t} &\sim \prod_{k=1 \atop k \equiv 1,2 (\text{mod}{4})} k,\text{ and}  \\ 
\prod_{t=1}^n B_{t}C_{t} &\sim \prod_{k=1 \atop {k \equiv 2,3 \atop (\text{mod}{4})}}^n (\alpha+k)(\beta+k) \prod_{k=1 \atop {k \equiv 0 \atop (\text{mod}{4})}}^n (\alpha + \beta + k) \prod_{k=n+1 \atop {k \equiv 3 \atop (\text{mod}{4})}}^{2n} (\alpha + \beta + k).
\end{align}
Together, these simplify the calculation of  $\left( -1, \prod_{t=1}^n \Delta_{t}\right)$.  We now turn to the calculation of $\prod_{t=1}^n \left( \Delta_{t-1},\Delta_{t} \right)_p$.

By \cite[Equation (6)]{hajir}, we have
\begin{align} \label{delta_2k}
\prod_{t=1}^n (\Delta_{t-1},\Delta_{t})_p = \prod_{k=1}^{(n-1)/2} (\Delta_{2k},\Delta_{2k-1}\Delta_{2k+1})_p.
\end{align}
It is easily verified from Equations (\ref{ADEF}), (\ref{BDEF}), and (\ref{CDEF}) that
\begin{align}
&A_{2\ell-1}A_{2\ell+1}\sim(n-2\ell),\label{A_odd_id}\\ 
&B_{2\ell-1}B_{2\ell+1}\sim(\alpha+n-2\ell+1)(\beta+n-2\ell+1)(\alpha+\beta+n-2\ell+2),\label{B_odd_id}\\ 
&C_{2\ell-1}C_{2\ell+1}\sim(\alpha+\beta+2n-4\ell+1)(\alpha+\beta+2n-4\ell+3),\label{C_odd_id}
\end{align}
for all $\ell\le (n-1)/2$, and
\begin{align}    
&A_{2\ell}\sim A_{2(\ell-1)}(n-2\ell+1), \label{A_even_id}\\
&B_{2\ell}\sim B_{2(\ell-1)}(\alpha+n-2\ell+2)(\beta+n-2\ell+2)(\alpha+\beta+n-2\ell+3), \label{B_even_id} \\
&C_{2\ell}\sim C_{2(\ell-1)}(\alpha+\beta+2n-4\ell+5)(\alpha+\beta+2n-4\ell+3),\text{and} \label{C_even_id}
\end{align}
for all $\ell\le n/2$. We similarly note that 
\begin{align}
&A_2B_2C_2\sim (n-1)(\alpha+n)(\beta+n)(\alpha+\beta+2n-1). \label{A2B2C2}
\end{align}
Thus using Theorem \ref{main_thm} and Equations (\ref{A_odd_id}), (\ref{B_odd_id}), and (\ref{C_odd_id}), we have
\begin{align} \label{delta_2k_with_ABC}
    \prod_{k=1}^{(n-1)/2} (\Delta_{2k},\Delta_{2k-1}\Delta_{2k+1})_p\sim  \prod_{k=1}^{(n-1)/2}(A_{2k}B_{2k}C_{2k},\,(n-2k)(\alpha+n-2k+1)(\beta+n-2k+1) \notag
    \\\cdot(\alpha+\beta+n-2k+2)(\alpha+\beta+2n-4k+1)(\alpha+\beta+2n-4k+3))_p. 
\end{align}
From Equations (\ref{A_even_id}), (\ref{B_even_id}), and (\ref{C_even_id}) we have
\begin{align} \label{A2kB2kC2k}
    A_{2k}B_{2k}C_{2k}\sim A_2B_2C_2\prod_{\ell=2}^k\big[(n-2\ell+1)(\alpha+n-2\ell+2)(\beta+n-2\ell+2)(\alpha+\beta+n-2\ell+3) \notag
    \\\cdot(\alpha+\beta+2n-4\ell+5)(\alpha+\beta+2n-4\ell+3) \big]
\end{align}
for all $k\in\{1,\ldots,\lfloor(n-1)/2\rfloor\}$. For ease of notation, we define
\begin{align} \label{j(ell)}
    j(\ell)=(n-2\ell+1)(\alpha+n-2\ell+2)(\beta+n-2\ell+2)(\alpha+\beta+n-2\ell+3)
    \\\cdot(\alpha+\beta+2n-4\ell+5)(\alpha+\beta+2n-4\ell+3) \notag
\end{align}
and
\begin{align} \label{m(ell)}
    m(\ell)=(n-2\ell)(\alpha+n-2\ell+1)(\beta+n-2\ell+1)(\alpha+\beta+n-2\ell+2)
    \\\cdot(\alpha+\beta+2n-4\ell+1)(\alpha+\beta+2n-4\ell+3) \notag
\end{align}
for all $\ell\in \mathbf{N}$. Substituting Equations (\ref{A2kB2kC2k}), (\ref{j(ell)}), and (\ref{m(ell)}) into Equation (\ref{delta_2k_with_ABC}) gives
\begin{align}
\prod_{k=1}^{(n-1)/2} (\Delta_{2k},\Delta_{2k-1}\Delta_{2k+1})_p&\sim \prod_{k=1}^{(n-1)/2}(A_2B_2C_2\prod_{\ell=2}^kj(\ell),\,m(k))_p \notag\\
&\sim \bigg(A_2B_2C_2,\,\prod_{k=1}^{(n-1)/2}m(k)\bigg)_p \cdot\prod_{k=2}^{(n-1)/2}\bigg(\prod_{\ell=2}^kj(\ell),\,m(k)\bigg)_p \notag\\
&\sim \bigg(A_2B_2C_2,\,\prod_{k=1}^{(n-1)/2}m(k)\bigg)_p\cdot\prod_{\ell=2}^{(n-1)/2}\bigg(j(\ell),\prod_{k=\ell}^{(n-1)/2}m(k)\bigg)_p \label{delta_2k_rearranged}
\end{align}
using properties of the Hilbert symbol. 

\subsection{Special Cases} \label{special_cases_section} We conclude the paper by giving some sample applications of our method.  First, it is an interesting problem to determine polynomials with specified Galois properties that are also ramified at few primes.  In \cite{hajir}, Hajir considered the polynomials
\[
f_n(x) = \sum_{j=0}^n \frac{n-j+1}{j!}x^j,
\]
and showed that they have Galois group $A_n$, which embeds into an $\widetilde{A}_{n}$-extension of $\Q$, whenever $n \equiv 1 \pmod{8}$.  We now give two explicit examples when $n=5$.  

\begin{exm}
Let $n=5$ and $(\alpha,\beta) = (-18,24)$.  Then we check that $\mathscr{J}_5^{(-18,24)}(x)$ is irreducible with Galois group $A_5$, and $\disc(\mathscr{J}_5^{(-18,24)}(x)) = 29^4/2^{22}$.  Using the explicit formulas of the previous section, we find that
\begin{align*}
\epsilon_p(E) &= (-3\cdot5\cdot13\cdot29,-2\cdot5\cdot11\cdot13)_p \cdot (3\cdot5\cdot11\cdot29,-2\cdot11\cdot13)_p \cdot (-1,2\cdot 5)_p \\
&=\begin{cases} -1 & \text{ if } p \in \lbrace 3,\infty \rbrace, \text{ and} \\
1 & \text{ otherwise.}
\end{cases}
\end{align*}
By Serre's theorem, $E$ does not embed into a quadratic extension with Galois group $\widetilde{A}_5$.
\end{exm}

\begin{exm}
Let $n=5$ and $(\alpha,\beta) = (-29/2,5)$.  Then we check that $\mathscr{J}_5^{(-29/2,5)}(x)$ is irreducible with Galois group $A_5$.  Using the explicit formulas of the previous section, one checks that $\epsilon_p(E) = 1$ for all primes $p$, including $p=\infty$. By Serre’s theorem the splitting field of $\mathscr{J}_5^{(-29/2,5)}(x)$ embeds inside an $\widetilde{A}_5$-extension.
\end{exm}

Finally, we show how to give an alternate proof of Feit's formula for $\Delta_t$ using our approach.  Recall from \cite{feit} the definition of the Generalized Laguerre Polynomials $F_n^{(\lambda)}(x)$:
\[
F_n^{(\lambda)}(x) = \sum_{j=0}^n (-1)^{n-j}\binom{n}{j} \left(\prod_{i=j+1}^n \lambda + i\right)x^j.
\]
Let $\gamma_j^{(\lambda)}(n) = (-1)^{n-j}\binom{n}{j} \prod_{k=j+1}^n(\lambda+k)$.  Then the generating function $G(z)$ is given by
\[
G(z) = \frac{\sum_{j=1}^n j\gamma_j^{(\lambda)}(n)z^{n-j}}{\sum_{j=0}^n \gamma_j^{(\lambda)}(n)z^{n-j}}.
\]

Proposition \ref{cont_frac_prop} does not use any properties specific to Jacobi Polynomials, so we can turn to the key Proposition \ref{f(n-l)_prop}; those formulas must be replaced by 
\begin{align*}
f_{2i}(n-\ell) &= f_{2i}(n-\ell+1) \cdot -(n-i-\ell)(\lambda + n -i-\ell+1),\text{ and} \\
f_{2i+1}(n-\ell) &= f_{2i+1}(n-\ell+1) \cdot -(n-i-\ell)(\lambda + n -i-\ell).
\end{align*}
Following the same sequence of calculations, one finds that $a_0=n$, $a_i = -(\lambda+n)$, and 
\begin{align*}
a_{2p} &=p-n \\
a_{2p+1} &= -(\lambda+n-p). 
\end{align*}
Then 
\begin{align*}
\Delta_t &= s_0^t (a_1a_2)^{t-1} \cdots (a_{2n-3}a_{2n-2}) \\
&=\prod_{k=n-t+1}^n k^{k-(n-t)} \prod_{k=n-t+2}^n (\lambda+k)^{k-(n-t+1)} \\
&=\prod_{j=n-t+1}^n j(j\lambda + j^2)^{j-n+t-1},
\end{align*}
which is exactly Feit's formula in  \cite[Thm.~5.1]{feit}.

\end{document}